\documentclass{amsart}
\usepackage{amsmath}
\usepackage{amssymb}
\usepackage{color}
\usepackage{graphicx}

\newcommand{\Der}{\operatorname{D^b}}
\newcommand{\gldim}{\operatorname{gldim}}
\newcommand{\End}{\operatorname{End}}
\newcommand{\Ext}{\operatorname{Ext}}
\newcommand{\rad}{\operatorname{rad}}
\newcommand{\cyc}{{\operatorname{cyc}}}
\newcommand{\Clu}{\mathcal{C}}
\newcommand{\Tensor}{\operatorname{T}}
\newcommand{\dual}{\operatorname{D}}
\newcommand{\gen}[1]{\langle #1\rangle}
\newcommand{\ealg}[1]{\operatorname{R}(#1)}
\newcommand{\calg}[1]{\operatorname{C}(#1)}
\newcommand{\standard}[1]{S(#1)}

\newtheorem{theorem}{Theorem}[section]
\newtheorem{theoremunotres}{Theorem 1.4}

\newtheorem{lemma}[theorem]{Lemma}
\newtheorem{proposition}[theorem]{Proposition}
\newtheorem{corollary}[theorem]{Corollary}
\theoremstyle{definition}
\newtheorem{example}[theorem]{Example}
\newtheorem{definition}[theorem]{Definition}
\newtheorem{remark}[theorem]{Remark}{

\renewcommand{\AA}{\mathbb{A}}
\newcommand{\DD}{\mathbb{D}}
\newcommand{\EE}{\mathbb{E}}

\newcommand{\HVCenter}[1]{\setbox 0=\hbox{#1}%
        \dimen0=\wd0%
        \dimen1=\ht0%
        \divide\dimen0 by 2%
        \divide\dimen1 by 2%
        \hskip -\dimen0%
        \lower \dimen1%
        \box0%
        \hskip -\dimen0}
\newcommand{\HBCenter}[1]{\setbox 0=\hbox{#1}%
        \dimen0=\wd0%
        \dimen1=\ht0%
        \divide\dimen0 by 2%
        \hskip -\dimen0%
        \box0%
        \hskip -\dimen0}
\newcommand{\HTCenter}[1]{\setbox 0=\hbox{#1}%
        \dimen0=\wd0%
        \dimen1=\ht0%
        \divide\dimen0 by 2%
        \hskip -\dimen0%
        \lower \dimen1%
        \box0%
        \hskip -\dimen0}
\newcommand{\RTCenter}[1]{\setbox 0=\hbox{#1}%
        \dimen0=\wd0%
        \dimen1=\ht0%
        \hskip -\dimen0%
        \lower \dimen1%
        \box0%
        \hskip -\dimen0}
\newcommand{\LTCenter}[1]{\setbox 0=\hbox{#1}%
        \dimen0=\wd0%
        \dimen1=\ht0%
        \lower \dimen1%
        \box0%
        \hskip -\dimen0}
\newcommand{\RVCenter}[1]{\setbox 0=\hbox{#1}%
        \dimen0=\wd0%
        \dimen1=\ht0%
        \divide\dimen1 by 2%
        \hskip -\dimen0%
        \lower \dimen1%
        \box0%
        \hskip -\dimen0}
\newcommand{\RBCenter}[1]{\setbox 0=\hbox{#1}%
        \dimen0=\wd0%
        \dimen1=\ht0%
        \hskip -\dimen0%
        \box0%
        \hskip -\dimen0}
\newcommand{\LVCenter}[1]{\setbox 0=\hbox{#1}%
        \dimen1=\ht0%
        \divide\dimen1 by 2%
        \lower \dimen1%
        \box0%
        \hskip -\dimen0}
\newcommand{\myjoin}{\begin{picture}(20,6)\put(3,3){\line(1,0){13
}}\end{picture}}

\parindent=0pt
\parskip=4pt

\begin{document}

\title{Cluster tilted algebras with a cyclically oriented quiver}
\author{Michael Barot}
\email{barot@matem.unam.mx}
\address{Instituto de Matem\'aticas, Universidad Nacional Aut\'onoma
de M\'exico, Ciudad Universitaria, C.P. 04510, Distrito Federal, Mexico.}
\author{Sonia Trepode}
\email{strepode@mdp.edu.ar}
\address{Departamento de Matemática, Facultad de Ciencias Exactas y Naturales, Funes 3350, Universidad Nacional de Mar del Plata, 7600 Mar del Plata, Argentina.}

\thanks{
The second author is a researcher from CONICET, Argentina. The
authors acknowledge partial support from the collaboration
project Argentina--M\'exico
MINCyT--CONACyT MX 0702, from CONICET,
Argentina and CONACyT, M\'exico. Both authors also would like to thank
Elsa Fern\'andez for her warm reception at Universidad Nacional de la
Patagonia San Juan Bosco, Puerto Madryn, Argentina.}

\begin{abstract}

In association with a finite dimensional algebra A of global
dimension two, we consider the endomorphism algebra of A, viewed as
an object in the triangulated hull of the orbit category of the
bounded derived category, in the sense of Amiot. We
characterize the algebras A  of global dimension two such that its
endomorphism algebra is isomorphic to a cluster-tilted algebra with
a cyclically oriented quiver. Furthermore, in the case that the
cluster tilted algebra with a cyclically oriented quiver is of
Dynkin or extended Dynkin type then A is derived equivalent to a hereditary
algebra of the same type.
\end{abstract}

\maketitle

\section{Introduction}

Let $A$ be a finite-dimensional algebra over an algebraically closed
field $k$. We denote by $\Der(A)$ the bounded derived category of
the category of finite-dimensional (left) $A$-modules. We denote by
$\tau$ the Auslander-Reiten translation and by $S$ the suspension of
$\Der(A)$.

Amiot showed in \cite{Amiot} that if the global dimension of $A$ is
less than or equal to two then the orbit category
$\Der(A)/\tau^{-1}S$ can be embedded fully faithfully in a
triangulated category $\Clu_A$, called the \emph{cluster category}
of $A$. This embedding is an equivalence if $A$ is a hereditary
algebra. In any case, $\End_{\Clu_A}(A)$ is isomorphic to the tensor
algebra $\calg{A}:=\Tensor_A(\Ext^2(\dual A,A))$. We say that the
algebra $A$ is \emph{derived equivalent} to the algebra $B$, if
$\Der(A)$ and $\Der(B)$ are triangle equivalent categories.
 If $A$ is derived
equivalent to a hereditary  algebra $H=k Q$ then the algebra $\calg{A}$ is
called \emph{cluster-tilted} algebra of type $Q$.

We say that a quiver $Q$ is  \emph{cyclically oriented}, if each
\emph{chordless} cycle is cyclically oriented, see \cite{BGZ} and
Section~\ref{sec:cyc-or}. We say that two paths $\gamma$ and
$\delta$ in a quiver are \emph{parallel} (resp. \emph{antiparallel})
if the have the same start $s(\gamma)=s(\delta)$ and end vertex
$e(\gamma)=e(\delta)$ (resp. if $e(\gamma)=s(\delta)$ and
$s(\gamma)=e(\delta)$). It follows from \cite{BGZ} and \cite{BRS}
that cluster-tilted algebras of Dynkin type $Q$ are characterized by
the fact that all the quivers in the mutation class of the quiver
$Q$, in the sense of \cite{FZ}, are cyclically oriented.

In this article we consider the problem of characterizing the
algebras $A$ of global dimension two, having $\calg{A}$ isomorphic
to a cluster-tilted algebra $C$. We solve the problem when the
the quiver $Q_C$ of $C$ is cyclically oriented. For solving
this problem, it was necessary to give an explicit description of
the defining relations of the class of cluster-tilted algebras with
a cyclically oriented quiver. This description generalizes the
result proved in \cite{BMR2} in  the case that $C$ is a
cluster-tilted of finite representation type.
We recall that a relation $\rho$ is called \emph{minimal} if whenever $ \rho= \sum_i \beta_i \rho_i\gamma_i $ where $\rho_i$ is a relation for every $i$, then
$\beta_i $ and $\gamma_i$ are scalars for some index $i$, (see \cite{BMR2}).

\begin{proposition}
\label{thm:cy-or-clu-tilt} If $C$ is a cluster-tilted algebra of any type whose
quiver $Q_C$ is cyclically oriented, then to each arrow $\alpha$,
belonging to an oriented cycle, the sum $\rho_\alpha$ of all paths
which are antiparallel to $\alpha$ is a  for $C$.
Moreover these are all minimal relations for $C$.
\end{proposition}

We obtain this result as consequence of Proposition~\ref{prop:rel-correspond-arrow}
and Proposition~\ref{cor:standard}.

Given a quiver $Q$, a subset $\Sigma$ of the set of arrows is called
an \emph{admissible cut} if $\Sigma$ contains exactly one arrow of
each chordless cycle in $Q$ which is oriented. The notion of
admissible cuts was firstly introduced in \cite{F} and \cite{FP} as
\emph{cutting sets}. The quotient obtained by deleting these arrows
is called \emph{quotient by an admissible cut}, see section
\ref{sec:prel} for a precise definition.

In this work we show that if $C$ is a cluster-tilted algebra whose
quiver $Q_C$ is cyclically oriented, then $C$ admits an admissible
cut. Even more, each arrow of $Q_C$, contained in an oriented cycle,
is also contained in an admissible cut, see
\ref{prop:adm-cut-exists}. We provide a necessary and sufficient
condition over $A$ such that $\calg{A}$ is isomorphic to a cluster
tilted algebra with a cyclically oriented quiver.  We are in a
position to state now our main result.

\begin{theorem}
\label{thm:adm-cut-clu-tilt} Let $A$ be an algebra with $\gldim
A\leq 2$, such that $\calg{A}$ is a finite dimensional algebra, and
let $C$ be a cluster-tilted algebra of any type with a cyclically oriented
quiver. Then, $\calg{A}\simeq C$ if and only if $A$ is the quotient
of $C$ by an admissible cut.
\end{theorem}

This result is shown in Section 4.

We consider the following question for a given quiver $Q$  and $T$ a
cluster-tilting object in $\Clu_{k Q}$: is it true for an algebra $A$ with
$\gldim A\leq 2$ that $\End_{\Clu_A}(A)\simeq
\End_{\Clu_{k Q}}(T)$ implies that $A$ is
 derived equivalent to $kQ$? This question has a negative answer as
 shows the following example.

\begin{example}
\label{ex:3}
  Let $B=k Q_B/I_B$ (resp. $C=k Q_C/I_C$) be the quotient of the path
  algebra of the quiver $Q_B$ (resp. $Q_C$) as shown in the following
  picture on the left (resp. right) and $I_B=\gen{\gamma\varphi}$
  (resp. $I_C=\gen{\varphi\eta}$).
  \begin{center}
    \begin{picture}(140,68)
      \put(0,8){
        \put(-20,50){$Q_B:$}
        \multiput(0,0)(40,0){2}{\circle*{3}}
        \multiput(0,40)(40,0){2}{\circle*{3}}
        \put(4,0){\vector(1,0){32}}
        \put(0,4){\vector(0,1){32}}
        \put(36,40){\vector(-1,0){32}}
        \put(3,37){\vector(1,-1){34}}
        \put(20,-4){\HTCenter{\small $\alpha$}}
        \put(-4,20){\RVCenter{\small $\beta$}}
        \put(22,22){\small $\gamma$}
        \put(20,44){\HBCenter{\small $\varphi$}}
      }
      \put(100,8){
        \put(-20,50){$Q_C:$}
        \multiput(0,0)(40,0){2}{\circle*{3}}
        \multiput(0,40)(40,0){2}{\circle*{3}}
        \put(4,0){\vector(1,0){32}}
        \put(0,4){\vector(0,1){32}}
        \put(36,40){\vector(-1,0){32}}
        \put(40,4){\vector(0,1){32}}
        \put(20,-4){\HTCenter{\small $\alpha$}}
        \put(-4,20){\RVCenter{\small $\beta$}}
        \put(20,44){\HBCenter{\small $\varphi$}}
        \put(44,20){\LVCenter{\small $\eta$}}
      }
    \end{picture}
  \end{center}
  Then $\Clu_C\simeq \Clu_B\simeq \Clu_{k \tilde{A}_{3,1}}$ but
  $\Der(B)\not\simeq \Der(C)$.
\end{example}

In the following, we want to show that there are interesting classes
of hereditary algebras $H= kQ$ for which the answer is always
affirmative.  In the case that $C$ is a cluster-tilted algebra of
Dynkin or extended Dynkin type $\Delta$ with a cyclically
oriented quiver, we show that if the algebra $\calg{A}$ is isomorphic to
$C$ then $A$ must necessarily be derived equivalent to a hereditary
algebra of type $\Delta$. More than that, we have the following
statement which provides a necessary and sufficient condition.

\begin{theoremunotres}
Let $A$ be a finite-dimensional algebra with $\gldim
  A\leq 2$, such that $\calg{A}$ has a cyclically oriented quiver.
  Then, $\calg{A}$ is cluster-tilted of Dynkin or extended Dynkin
  type $\Delta$ if and only if $A$ is derived equivalent to a hereditary algebra
  $H$ of Dynkin or extended Dynkin type $\Delta$.
\end{theoremunotres}
\stepcounter{theorem}

A closely related problem was considered by  C.~Amiot and S.~Oppermann
in  \cite{AO,AO2}. They studied when two algebras of global dimension 2 give rise to the same cluster category, and under which assumptions they become derived equivalent algebras. The authors give an answer to this question in terms of Galois coverings.

\section{Preliminary results}
\label{sec:prel}
For $A$ an algebra, let $\ealg{A} = A\ltimes \Ext^2_A(\dual A,A)$ be the \emph{relation extension} of $A$, see also \cite{ABS}.

\begin{lemma}
  \label{lem:D1}
  Let $A$ be a finite-dimensional algebra with $\gldim A\leq 2$ and
  denote $\calg{A}=\End_{\Clu_A}(A)$.
Then, there is a sequence of algebra homomorphisms
  $$
  A\xrightarrow{\iota}\calg{A}\xrightarrow{\pi}
  \ealg{A}\xrightarrow{\mu} A
  $$
  whose composition is the identity.  If $\calg{A}$ is a finite-dimensional algebra,
  then $\calg{A}$ is a split extension of $A$.

\end{lemma}

\begin{proof}
  Observe first that $A\simeq \End_{\Der(A)}(A)$ is naturally embedded
  in $\End_{\Phi}(A)$, where $\Phi=\Der(A)/F$ is the orbit
  category. Since the embedding of $\Phi$ in $\Clu_A$ is fully
  faithful, we have $\End_{\Phi}(A)\simeq \End_{\Clu_A}(A)$. This
  defines the homomorphism $\iota$. Since $C=\End_{\Clu_A}(A)$ is
  naturally isomorphic to the tensor algebra $\Tensor_A(\Ext^2(\dual
  A,A))$, we have that $\ealg{A}$ is a natural quotient of $C$. And
  hence so is $A$ of $\ealg{A}$.
\end{proof}

We recall the following definition from \cite{BFPPT}.

\begin{definition}[Quotient by an admissible cut]
  Let $C=k Q_C/I$ be an algebra given by a quiver $Q_C$ and an
  admissible ideal $I$. A \emph{quotient of $C$ by an admissible cut}
  (or \emph{an admissible cut of $C$)} is an algebra of the form
  $k Q_C/\gen{I\cup \Delta}$ where $\Delta$ is an admissible cut
  of $Q_C$.
\end{definition}

\begin{lemma}
  \label{lem:cut-iff}
  Let $A$ be a finite-dimensional algebra with $\gldim A\leq 2$ and such that $\calg{A}$ is finite-dimensional.
  Then $A$ is the quotient of $\ealg{A}$ by an admissible cut if and only if $A$ is the quotient of $\calg{A}$ by
  an admissible cut.
\end{lemma}

\begin{proof}
  The proof is done similarly as in \cite[Remark 4.15]{BFPPT}. The proof of both implications is done simultaneously. Suppose
  that $A$ is an admissible cut of $B=\ealg{A}$ (resp. of $B=\calg{A}$). Then by
  Lemma~\ref{lem:D1} the algebras $\ealg{A}$ and $\calg{A}$ are both split
  extensions of $A$ and by \cite[Proposition 4.16]{Amiot} they have the same
  quiver. Hence the quiver $Q_{A}$ is obtained from the quiver
  $Q_B$ by some admissible cut $\Delta$.

  Let $J$ be the ideal of $k Q_{B}$ such that $A=(k
  Q_{B})/J$. Clearly, we have $I_B\cup\Delta\subseteq
  J$  and hence $\gen{I_B\cup \Delta}\subseteq J$. We now
  prove that also the converse contention holds. Let $\rho$ be a
  relation of $k Q_B$  which does belong to $J$. Write
  $\rho=\sum_{i=1}^t \lambda_i\rho_i$ for some non-zero scalars
  $\lambda_1,\ldots,\lambda_t$ and some parallel paths
  $\rho_i=\rho_{i,N_i}\cdots\rho_{i,1}$ for $i=1,\ldots,t$. If
  $\rho_{i,j}$ belongs to $\Delta$ we have that
  $\rho'=\rho-\lambda_i\rho_i$ belongs to $J$ and hence inductively
  over the number of summands $t$ we obtain that $\rho'$ (and hence
  also $\rho$) belongs to $\gen{I_B\cup \Delta}$. So it
  remains to consider the case where no path $\rho_i$ contains an
  arrow of $\Delta$.

  Then $\rho$ can be considered as an element $k Q_A$. Let $\pi$ and
  $\mu$ be the canonical maps as mentioned in Lemma~\ref{lem:D1}. Then, in case we supposed that $A$ is the quotient of $\calg{A}$ by an admissible cut we argue as follows: we have
  $\overline{\rho}=\pi{\overline{\rho}}$, where
  $\overline{\rho}$ denotes both the class of $\rho$ in the quotient
  $k Q_{\calg{A}}/I_{\calg{A}}$ and $k Q_A/I_A$. Then
  $0=\mu(\overline{\rho})=\mu\pi(\overline{\rho})=\overline{\rho}$
  shows that indeed $\rho$ belongs to $I_{\calg{A}}$. This shows that
  $A$ is an admissible cut of $\calg{A}$.

  In case we supposed that $A$ is the quotient of $\ealg{A}$ by an admissible cut then $\mu(\rho)=0$ shows that $\rho=\mu\pi\iota(\rho)=0$ showing that $\rho\in I_{A}\subseteq I_{\calg{A}}$. Again, we have that $A$ is the quotient of $\calg{A}$.
\end{proof}

\section{Cyclically oriented quivers}
\label{sec:cyc-or}

\subsection{Shortest paths and chordless cycles}

We recall from \cite{BGZ} the following definitions.

\begin{definition}
A \emph{walk of length $p$} in a quiver $Q$ is a $(2p+1)$-tuple
$$
w=(x_p,\alpha_p,x_{p-1},\alpha_{p-1},\ldots,x_1,\alpha_1,x_0)
$$
such that for all $i$ we have $x_i\in Q_0$, $\alpha\in Q_1$ and $\{s(\alpha_i),e(\alpha_i)\}=\{x_p,x_{p-1}\}$.
The walk $w$ is \emph{oriented} if either $s(\alpha_i)=x_{p-1}$ and $e(\alpha_i)=x_p$ for all $i$ or $s(\alpha_i)=x_{p}$ and $e(\alpha_i)=x_{p-1}$ for all $i$.
Furthermore, $w$ is called a \emph{cycle} if $x_0=x_p$. A cycle of length $1$ is called a \emph{loop}.
We often omit the vertices and abbreviate $w$ by $\alpha_p\cdots\alpha_1$. An oriented walk is also called \emph{path}.

A cycle $c=(x_p,\alpha_p,\ldots,x_1,\alpha_1,x_p)$ is called \emph{non-intersecting} if its vertices $x_1,\ldots,x_p$ are pairwise distinct.
A non-intersecting cycle of length $2$ is called $2$-cycle.
If $c$ is a non-intersecting cycle then any arrow $\beta\in Q\setminus\{\alpha_1,\ldots,\alpha_p\}$ with $\{s(\beta),e(\beta)\}\subseteq\{x_1,\ldots,x_p\}$ is called a \emph{chord} of $c$. A cycle $c$ is called \emph{chordless} if it is non-intersecting and there is no chord of $c$.

A quiver $Q$ without loop and $2$-cycle is call \emph{cyclically oriented} if each chordless cycle is
oriented. Note that this implies that there are no multiple arrows in $Q$. A quiver without oriented cycle is called \emph{acyclic} and an algebra whose quiver is acyclic is called \emph{triangular}.
\end{definition}

\begin{remark}
\label{rem:C_n}
The easiest cyclically oriented quiver is clearly a single oriented cycle.
We denote by $C_n$ the cyclically oriented cycle with $n$ vertices.
Observe that for each $n$ there exists a cluster-tilted algebra $A$ having a
quiver isomorphic to $C_n$, namely the algebra $k C_n/\rad^{n-1}$ where
$\rad^{n-1}$ is the ideal generated by all paths of length $n-1$. By \cite{BMR2}, there exists, up to isomorphism only one cluster-tilted
algebra with predefined quiver and hence we always must have all
compositions of $n-1$ arrows to be the minimal relations of $A$ whenever
$A$ is cluster-tilted and $Q_A=C_n$.
\end{remark}

\begin{definition}
  A path $\gamma$ which is antiparallel to an arrow $\eta$ in a quiver $Q$ is a \emph{shortest path} if
the full subquiver generated by the induced oriented cycle $\eta\gamma$ is chordless.
A path $\gamma=(x_0\xrightarrow{\gamma_1} x_1 \xrightarrow x_2 \rightarrow \cdots
   \rightarrow x_L)$ is called \emph{shortest directed path} if there exists no arrow $x_i\rightarrow x_j$ in $Q$ with $1\leq i+1<j\leq L$.
   A walk $\gamma=(x_0\myjoin x_1 \myjoin x_2 \myjoin \cdots
   \myjoin x_L)$ is called a \emph{shortest walk} if there is no edge joining $x_i$ with $x_j$ with $1\leq i+1<j\leq L$ and $(i,j)\neq(0,L)$ (we write a horizontal line to
indicate an arrow oriented in one of the two possible ways).
\end{definition}

\begin{lemma}
  \label{lem:shortest-path}
  Let $Q$ be a cyclically oriented quiver.
   Suppose that $\gamma$
    is a shortest directed path in a quiver $Q$ which is antiparallel to some arrow $\eta$.
    Then $\eta\gamma$ is a chordless cycle. Conversely if
    $\eta\gamma$ is an oriented chordless cycle, then $\gamma$ is a
    shortest directed path.
\end{lemma}

\begin{proof}
Let
  $$\gamma=(x_0\xrightarrow{\gamma_1} x_1 \xrightarrow{\gamma_2} x_2 \rightarrow \cdots
   \rightarrow x_{L-1}\xrightarrow{\gamma_L} x_L)$$
   and denote by $Q'$ the full subquiver of $Q$ given by the vertices of
   $x_0,\ldots,x_L$.
   Suppose now that there exists a chord in the cycle $\eta\gamma$.
   Let $r$ be maximal with $1<r\leq L$ such that there exists a chord
  $x_r\rightarrow x_s$ in $Q'$ with $s<r$. Choose such a chord
  such that $s$  is minimal. Then look at the full subquiver $S^{(1)}$ given
  by the vertices $x_1,\ldots,x_s,x_r\ldots,x_L$. Now there is no arrow connecting a vertex $v_i$
  with $i\geq r$ with a vertex $v_j$ with $j\neq i\pm 1$.
  Thus if there exists an arrow $x_i\rightarrow x_j$ in $S^{(1)}$ for some $j<i$ then $i\leq s$. Take $i$ maximal and then $j$ minimal and look at the subquiver$S^{(2)}$ given by the vertices
  $x_1,\ldots,x_j,x_i,\ldots,x_s,x_r\ldots,x_L$. Inductively we obtain a non-oriented walk $\delta$ which forms a non-oriented chordless cycle with $\eta$, in contradiction to the hypothesis on $Q$.

 The converse statement follows immediately from the definitions.
\end{proof}

\begin{proposition}
\label{prop:cyc-or-quiv} Let $Q$ be a cyclically oriented quiver.
Then for any arrow $\eta$, which belongs to an oriented cycle, each
two distinct shortest paths antiparallel to $\eta$ share only the
starting vertex and the end vertex. Hence the diagram of all cycles
containing $\eta$ looks as follows.
     \begin{equation}\label{eq:diag-antiparallel}
       \begin{picture}(143,70)
         \put(0,0){
            \put(0,15){
                \multiput(30,0)(20,0){2}{\circle*{3}}
                \multiput(93,0)(20,0){2}{\circle*{3}}
             \multiput(34,0)(63,0){2}{\vector(1,0){12}}
             \put(54,0){\line(1,0){8}}
             \multiput(66,0)(5,0){3}{\line(1,0){1}}
             \put(81,0){\vector(1,0){8}}
           }
           \put(0,40){
                \multiput(30,0)(20,0){2}{\circle*{3}}
                \multiput(93,0)(20,0){2}{\circle*{3}}
             \multiput(34,0)(63,0){2}{\vector(1,0){12}}
             \put(54,0){\line(1,0){8}}
             \multiput(66,0)(5,0){3}{\line(1,0){1}}
             \put(81,0){\vector(1,0){8}}
           }
           \put(0,60){
                \multiput(30,0)(20,0){2}{\circle*{3}}
                \multiput(93,0)(20,0){2}{\circle*{3}}
             \multiput(34,0)(63,0){2}{\vector(1,0){12}}
             \put(54,0){\line(1,0){8}}
             \multiput(66,0)(5,0){3}{\line(1,0){1}}
             \put(81,0){\vector(1,0){8}}
           }
           \multiput(0,0)(143,0){2}{\circle*{3}}
           \put(3,1.5){\vector(2,1){24}}
           \put(1.8,2.4){\vector(3,4){26.4}}
           \put(1.5,3){\vector(1,2){27}}
           \put(113,0){
             \put(3,13.5){\vector(2,-1){24}}
             \put(1.8,37.6){\vector(3,-4){26.4}}
             \put(1.5,57){\vector(1,-2){27}}
           }
           \put(139,0){\vector(-1,0){135}}
           \put(71.5,63){\HBCenter{\small $\delta_1$}}
           \put(71.5,43){\HBCenter{\small $\delta_2$}}
           \put(71.5,18){\HBCenter{\small $\delta_t$}}
           \put(71.5,-3){\HTCenter{\small $\eta$}}
           \multiput(41.5,30)(60,0){2}{
              \multiput(0,0)(0,-3){3}{\circle*{1}}
            }
            \multiput(71.5,35)(0,-3){3}{\circle*{1}}
            \put(-3,0){\RVCenter{\small $x$}}
            \put(146,0){\LVCenter{\small $y$}}
         }
       \end{picture}
     \end{equation}
\end{proposition}

\begin{proof}
Let $\eta\colon y\rightarrow x$ and
$\delta_1=\alpha_m\cdots\alpha_1$ and
$\delta_2=\beta_n\cdots\beta_1$ be two shortest paths antiparallel
to $\eta$. Let $\varepsilon_1=\alpha_j\alpha_{j-1}\cdots\alpha_i$ be
a subpath of $\delta_1$ which is parallel to a subpath
$\varepsilon_2=\beta_\ell\cdots\beta_k$ of $\delta_2$.

Suppose that $\varepsilon_1$ is a proper subpath of $\delta_1$, that
is, $i>1$ or $j<m$. Then also $\varepsilon_2$ is a proper subpath of
$\delta_2$. Then
$$
z_i\xrightarrow{\alpha_i} z_{i+1}\rightarrow\cdots\rightarrow z_{j}
\xrightarrow{\alpha_j} z_{j+1}=z'_{\ell+1}
\xleftarrow{\beta_\ell}z_{\ell}\leftarrow\cdots\leftarrow
z'_{k+1}\xleftarrow{\beta_k}z'_k=z_i
$$
is a non-oriented cycle and hence by hypothesis not chordless.
If no chord would end in $z_i=z_k'$ then there would exist
a non-oriented chordless cycle containing the arrows $\alpha_i$ and $\beta_k$, in contradiction to the hypothesis.
Hence there must exist a chord ending in $z_i=z'_k$, say
$\varphi\colon z_h\rightarrow z_i$ for some $i+1<h\leq j+1$, in
contradiction to Lemma~\ref{lem:shortest-path}.
\end{proof}

\begin{proposition}
\label{prop:cyc-or-quiv2} Let $Q$ be a connected, cyclically oriented
quiver and $\delta_1,\ldots,\delta_t$ be the paths which are
antiparallel to an arrow $\eta\colon y\rightarrow x$. Further let
$Q'$ be the full subquiver of $Q$ given by the vertices $Q_0\setminus\{x,y\}$ and for each $i=1,\ldots,t$
let $\Gamma_i$ be the connected component of $Q'$ containing the
vertices of $\delta_i$ distinct from $x$ and $y$.

\begin{center}
\includegraphics{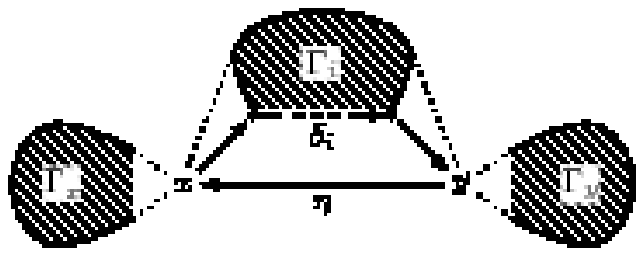}
\end{center}

Then $\Gamma_1,\ldots,\Gamma_t$ are pairwise disjoint subquivers of
$Q'$.
\end{proposition}

\begin{proof}
Suppose the contrary. Then there exists in $Q'$ a non-oriented walk
between $\delta_i\setminus\{x,y\}$ and $\delta_j\setminus\{x,y\}$
for some $i\neq j$. To fix notation, let
\begin{align*}
\delta_i&=(x=a_1\longrightarrow a_2\longrightarrow\cdots\longrightarrow a_m=y)\\
\delta_j&=(x=b_1\longrightarrow
b_2\longrightarrow\cdots\longrightarrow b_n=y)
\end{align*}
and let
$$
\sigma=(a_{m'}=c_1\myjoin c_2\myjoin\cdots\myjoin c_p=b_{n'})
$$
be a walk of shortest length $p$. Then
$$
x=a_1\longrightarrow\cdots\longrightarrow a_{m'} =c_1\myjoin
c_2\myjoin\cdots\myjoin c_p=b_{n'}\longleftarrow
\cdots\longleftarrow b_1=x
$$
is a non-oriented cycle and hence not chordless. However by
Proposition~\ref{prop:cyc-or-quiv} and the minimality of the length
of $\sigma$, no chord can exist between two points of
$\{a_2,\ldots,a_{m'},c_2,\ldots,c_{p-1},b_{n'},\ldots,b_2\}$.
Therefore all chords join $x$ with a vertex $c_{p'}$ with $1<p'<p$
and we may assume $p'$ to be minimal with that property.
Consequently
$$
x=a_1\longrightarrow\cdots\longrightarrow a_{m'} =c_1\myjoin
c_2\myjoin\cdots\myjoin c_{p'}\myjoin x
$$
is a chordless cycle and therefore oriented. This shows that the
edge $a_{m'}\myjoin c_2$ is oriented towards $c_2$, that is,
$a_{m'}\longrightarrow c_2$. Now,
$$
y=a_m\longleftarrow a_{m-1}\longleftarrow\cdots\longleftarrow
a_{m'}\longrightarrow c_2\myjoin\cdots\myjoin
c_p=b_{n'}\longrightarrow \cdots\longrightarrow b_n=y
$$
is a non-oriented cycle for which no chord can end in $a_{m'}$.
Hence $Q$ contains a non-oriented chordless cycle, a contradiction.
\end{proof}

\begin{proposition}
\label{prop:cyc-or-quiv3} Let $Q$ be a connected cyclically oriented
quiver and let $Q'$ and $\Gamma_1,\ldots,\Gamma_t$ be as in
Proposition~\ref{prop:cyc-or-quiv2}. Furthermore, let $Q''$ be the
subquiver of $Q$ obtained by deleting
$\Gamma_1\cup\ldots\cup\Gamma_t$ and also the arrow $\eta$ (but not
$x$ and $y$). Then define $\overline{\Gamma}_x$ (resp.
$\overline{\Gamma}_y$) to be the connected component of $Q''$ which
contains $x$ (resp. $y$). Furthermore, for $i=1,\ldots,t$, let
$\overline{\Gamma}_i$ be the subquiver of $Q$ obtained from the full
subquiver of $Q$ on the vertices $\Gamma_i\cup\{x,y\}$ by removing
the arrow $\eta$.

\begin{center}
\includegraphics{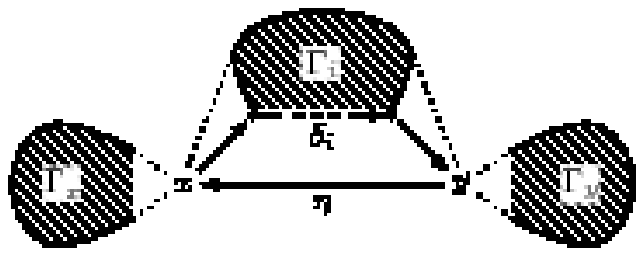}
\end{center}

Then each chordless cycle of $Q$ different from
$\eta\delta_1,\ldots,\eta\delta_t$ is contained in one of the
subquivers
$\overline{\Gamma}_x,\overline{\Gamma}_y,\overline{\Gamma}_1,\ldots,\overline{\Gamma}_t$.
\end{proposition}

\begin{proof}
Let $\Gamma_x=\overline{\Gamma}_x\setminus\{x\}$ and
$\Gamma_y=\overline{\Gamma}_y\setminus\{y\}$ as full subquivers of
$Q'$. Then $\Gamma_x$ and $\Gamma_y$ are disjoint in $Q'$. Otherwise
there would exist a connection
$$
\sigma=(x=c_1\myjoin c_2\myjoin \cdots\myjoin c_p=y)
$$
in $Q''$ with $c_2,\ldots,c_{p-1}\in\Gamma_x\cup \Gamma_y$. A minimal such connection
together with $\eta$ would be a chordless cycle, hence oriented and
therefore equal to one of the cycles
$\eta\delta_1,\ldots,\eta\delta_t$, a contradiction to the fact that $\Gamma_x,\Gamma_y$ are disjoint by definition from $\Gamma_i$ for any $i=1,\ldots,t$.

Therefore $\Gamma_x,\Gamma_y,\Gamma_1,\ldots,\Gamma_t$ are pairwise
disconnected in $Q'$. This shows that a chordless cycle which does
not contain $x$ or $y$ lies in one of these components of $Q'$.

Let $\gamma$ be a chordless cycle in $Q$ which contains the vertex $x$. If $\gamma$
contains also $y$ then it must contain $\eta$ and therefore it must be one of the cycles
$\eta\delta_1,\ldots,\eta\delta_t$.

Now suppose that $y$ is not a vertex of $\gamma$. Then we first suppose that $\gamma$ contains
one or more arrows of $\delta_i$. Thus, let
$$
\gamma=(x\xrightarrow{\alpha_1} x_1\rightarrow \cdots \rightarrow x_k \xrightarrow{\beta_k} y_{k+1} \rightarrow \cdots
\rightarrow y_m \xrightarrow{\beta_m} y_{m+1}=x)
$$
where $x,x_1,\ldots,x_k$ are vertices of $\delta_i$ but $y_k$ is not. Then, by definition $y_k$ belongs to $\Gamma_i$.
Inductively, we see that $y_h$ also belongs to $\Gamma_i$ or $y_h=x$ (since $y_h=y$ is impossible, by hypothesis),
showing that $\gamma$ is contained in $\overline{\Gamma}_i$.

It remains to consider the case where $x$ is the only vertex of \eqref{eq:diag-antiparallel} belonging to $\gamma$. In that case
we have
$$
\gamma=(x\xrightarrow{\beta_1} y_1\rightarrow \cdots \rightarrow y_m \xrightarrow{\beta_m} y_{m+1}=x)
$$
and very similar argument works depending whether $y_1$ belongs to $\Gamma_i$ for some $i$ or to $\Gamma_x$.
In any case, it follows inductively that all other vertices must also belong to the same component, showing that
$\gamma$ is contained in $\overline{\Gamma}_i$ or $\overline{\Gamma}_x$.

The case where $y$ belongs to $\gamma$ but $x$ is
not is handled completely similar.
\end{proof}

\subsection{Existence of admissible cuts}

\begin{proposition}
\label{prop:adm-cut-exists}
Let $Q$ be a cyclically oriented quiver.
Then for each arrow $\alpha$ which belongs to an oriented cycle,
there exists an admissible cut $\Sigma$ which contains $\alpha$.
\end{proposition}

\begin{proof}
Let $\gamma=\alpha_m\cdots\alpha_2\alpha_1$ be an oriented chordless
cycle with $\alpha_1=\alpha$ and set $\eta=\alpha_m$. Then let
$\delta_1,\delta_2,\ldots,\delta_t$ be
the shortest paths which are antiparallel to $\eta$. We assume without loss of generality that $\delta_1==\alpha_{m-1}\cdots\alpha_1$. To fix
notation let $\delta_i=\beta_{i,n_i}\cdots\beta_{i,1}$ for
$i=2,\ldots,t$. Then define
$\Sigma'=\{\alpha,\beta_{2,1},\ldots,\beta_{t,1}\}$.

Let $\eta\colon y\rightarrow x$ and define $Q'=Q\setminus\{x,y\}$
and $\Gamma_i$ for $i=1,\ldots,t$ as in Proposition
\ref{prop:cyc-or-quiv2} and $\overline{\Gamma}_i$ as in Proposition~\ref{prop:cyc-or-quiv3}. By induction on the
number of arrows, there exists an admissible cut $\Sigma_i$ in
$\overline{\Gamma}_i$ with $\beta_{i,1}\in\Sigma_i$ for
$i=1,\ldots,t$ (where $\beta_{1,1}=\alpha$). Furthermore, again by
induction hypothesis there exist admissible cuts $\Sigma_x$ of
$\overline{\Gamma}_x$ and $\Sigma_y$ of $\overline{\Gamma}_y$. By
Proposition~\ref{prop:cyc-or-quiv3}, the set
$
\Sigma_x\cup\Sigma_y\cup\Sigma_1\cup\ldots\cup\Sigma_t
$
is an admissible cut of $Q$.
\end{proof}

\begin{lemma}
\label{lem:no-bypass}
Let $Q$ be a cyclically oriented quiver and
$\Sigma$ an admissible cut of $Q$. Then the quiver $Q'$, obtained
from $Q$ by removing the arrows $\Sigma$, has no \emph{bypass}, that
is, an arrow parallel to a path.
\end{lemma}

\begin{proof}
Denote by $Q'=Q\setminus\Sigma$ the quiver obtained from $Q$ by
deleting the arrows $\Sigma$. We first show that no arrow $\eta$ in
$Q'$ is parallel to a shortest directed path $\gamma$ in
$Q'\setminus\{\eta\}$. Assume otherwise. Then $\gamma$ is not an
arrow since $Q$ has no multiple arrows. To fix the notation, let
$$
\gamma=(x_1\xrightarrow{\gamma_1}
x_2\longrightarrow\cdots\longrightarrow
x_{m-1}\xrightarrow{\gamma_{m-1}} x_m)
$$
Since $\eta\gamma$ is non-oriented there must exist a chord. In
fact, there must exist a chord in $Q\setminus Q'$ ending in $x_1$, since otherwise we would have a non-oriented chordless cycle in $Q$ containing the arrows $\eta$ and $\gamma_1$. Let $s_1$ be maximal
with $2<s_1<m$ such that there exists a chord $\beta_1\colon
x_{s_1}\rightarrow x_{1}$. Now, the cycle
$$
x_1\xleftarrow{\beta_1}
x_{s_1}\xrightarrow{\gamma_{s_1}}x_{s_1+1}\longrightarrow\cdots\longrightarrow
x_{m-1}\longrightarrow x_m\xleftarrow{\eta}x_1
$$
is non-oriented. Inductively we get a sequence of arrows which
together with $\eta$ form an oriented cycle
\begin{equation}
\label{eq:proof-no-bypass}
x_1=x_{s_0}\xleftarrow{\beta_1}x_{s_1}\xleftarrow{\beta_2}x_{s_2}
\longleftarrow \cdots\longleftarrow x_{s_{t}}\xleftarrow{\beta_t}
x_{s_{t+1}}=x_m\xleftarrow{\eta} x_1.
\end{equation}
This cycle is chordless: by the maximality the indices $s_j$ there
exists no chord $x_{s_i}\longleftarrow x_{s_j}$ for $i<j+1$ and
since $\gamma$ is a shortest directed path in $Q\setminus\{\eta\}$ there
exists no chord in the opposite direction either. This contradicts
the fact that $\Sigma$ is an admissible cut since
$\beta_1,\ldots,\beta_t\in Q_1\setminus Q_1'=\Sigma$ belong to the cycle
\eqref{eq:proof-no-bypass} and $t>1$ since $Q$ does not contain
$2$-cycles.

Now, if there is an oriented path $\gamma$ parallel to $\eta$ which is not a shortest directed path. Then there exists a shortest directed path $\gamma'$ parallel to $\eta$ and we are done by the previous argument.
\end{proof}

\begin{proposition}
\label{prop:cut-directed}
Let $Q$ be a cyclically oriented quiver and
$\Sigma$ an admissible cut of $Q$. Then the quiver $Q'$ obtained
from $Q$ by removing the arrows $\Sigma$, is directed, that is, each
cycle in $Q$ is non-oriented.
\end{proposition}

\begin{proof}
We first show that there doesn't exist an oriented cycle
$$
\gamma=(x_1\xrightarrow{\gamma_1}
x_2\longrightarrow\cdots\longrightarrow
x_{m-1}\xrightarrow{\gamma_{m-1}} x_m\xrightarrow{\gamma_{m}}x_1)
$$
in $Q'$ which is chordless. Assume otherwise. Then this cycle can
not be chordless in $Q$ since $\Sigma$ is an admissible cut. Let
$\delta_1\colon x_i\rightarrow x_j$ be a chord of $\gamma$.
After possibly renumbering the vertices of $\gamma$,
we may assume without loss of generality, that $i>j=1$ and that $i$ is maximal.
We then set $n_1=i$ and observe that
$$
x_1\xleftarrow{\delta_1}x_{n_1}\xrightarrow{\gamma_{n_1}}x_{n_1+1}\rightarrow\cdot\rightarrow x_x\xrightarrow{\gamma_{m}}x_1
$$
is an non-oriented cycle in $Q$ and therefore can not be chordless. Since $x_{n_1}$ is a source of this cycle there must exist a
chord $\delta_2\colon x_{n_2}\rightarrow x_{n_1}$ ending in $x_{n_1}$. Again we assume that $\delta_2$ is chosen such that $n_2$ is maximal. Proceeding this way we find an oriented cycle
$$
\delta=(x_1\xleftarrow{\delta_1}x_{n_1}\xleftarrow{\delta_2}x_{n_2}\leftarrow\cdots\leftarrow x_{n_t}\xleftarrow{\delta_t}x_1)
$$
whose arrows all belong to $\Sigma=Q_1\setminus Q'_1$.
Now, if the cycle $\delta$ is not chordless then there exists a
chord $\varepsilon\colon x_{n_1}\rightarrow x_{n_j}$ dividing $\delta$ into two cycles of
smaller length, one of them oriented the other non-oriented. Proceeding
with the former we get inductively an oriented chordless cycle in $Q$ which belongs to $\Sigma$
in contradiction to $\Sigma$ being an admissible cut.
\end{proof}

\section{Cluster-tilted algebras whose quiver is cyclically oriented}

In this section we are going to construct the minimal relations of a cluster tilted algebra using only its ordinary quiver.
  By the other hand, it follows by \cite{BIRS} Corollary 6.8, that the relations  in a cluster tilted algebra come from a potential, but in order to construct this potential, we would need the minimal relations of the tilted algebra which give rise to the cluster tilted algebra.  In our case we derive the relations directly from the ordinary quiver of the cluster tilted algebra.

\subsection{Killing of idempotents}

Let $C$ be a cluster-tilted algebra with quiver $Q$ and $e$ some idempotent
of $C$. By \cite{BMR}, the quotient $C/CeC$ is again cluster-tilted.
We shall use this result frequently when $e=\sum_{x\in I}e_x$ is the
sum of trivial paths of some vertices $I\subset Q_0$ and call the quotient
$C/CeC$ to be the algebra \emph{obtained by killing the vertices of $I$}.


\subsection{The homotopy relation}
Given an arrow $\alpha$, we denote by  $\alpha^{-1}$ its formal inverse
  . A \textit{walk} in $Q$ from $x$ to $y$ is a formal composition
  $\alpha^{\epsilon_1}_1\alpha^{\epsilon_2}_2 \cdots
\alpha^{\epsilon_t}_t$ from $x$ to $y$, where $\alpha_i \in Q$ and
$\epsilon_i \in \{1,-1\}$ for all i. The \textit{homotopy relation}
 is the smallest equivalence relation in the set of walks in $Q$ such that:

\begin{enumerate}
    \item[a)] For all  $\alpha:x \rightarrow y$ we have $\alpha\alpha^{-1} \sim
    e_x$ y $\alpha^{-1}\alpha \sim e_y$.

    \item[b)]For each minimal relation
    $\sum^{m}_{i=1}\lambda_iw_i$, we have $w_i \sim w_j,\  \forall\
    i,j$.

    \item[c)] If $u \sim v$, then $wuw' \sim wvw'$, wherever this products are defined.

\end{enumerate}

The set of the equivalence classes of the walks ending or starting in a fix point
 $x_0$ is a group, called the
el \textit{fundamental group of $(Q, I)$}, we denote by
$\pi_1(Q, I)$.

A triangular algebra A is called \textit{simply
connected} if, for any presentation $(Q_A, I)$ of A, the group
$\pi_1(Q_A, I)$ is trivial, \cite{BG}.
A full subquiver $Q'$ of $Q$ is called \emph{convex} if for any two
paths $\gamma$, $\delta$ with $e(\gamma)=s(\delta)$ and $s(\gamma),
e(\delta)\in Q'_0$ then $e(\gamma)\in Q'_0$.
If $Q'$ is a full subquiver of $Q$ we denote $e_{Q'}=\sum_{i\in Q'_0} e_i$ and $A_{Q'}=Ae_{Q'}A$.
An algebra $A$ with quiver $Q$ is called \emph{strongly simply
connected} if for every full and convex subquiver $Q'$ of $Q$ the algebra $A_{Q'}$ is simply
connected.

\subsection{Relations which are antiparallel to arrows}

\begin{proposition}
\label{prop:clu-tilt-canonical} Assume that $C$ is a cluster-tilted
algebra whose quiver contains an arrow $\eta$ which is antiparallel
to $t$ paths $\delta_1,\ldots,\delta_t$ which share only the
starting point and the end point as vertices, that is, the quiver of
$C$ looks as shown in the following picture.
     \begin{center}
       \begin{picture}(143,70)
         \put(0,0){
            \put(0,15){
                \multiput(30,0)(20,0){2}{\circle*{3}}
                \multiput(93,0)(20,0){2}{\circle*{3}}
             \multiput(34,0)(63,0){2}{\vector(1,0){12}}
             \put(54,0){\line(1,0){8}}
             \multiput(66,0)(5,0){3}{\line(1,0){1}}
             \put(81,0){\vector(1,0){8}}
           }
           \put(0,40){
                \multiput(30,0)(20,0){2}{\circle*{3}}
                \multiput(93,0)(20,0){2}{\circle*{3}}
             \multiput(34,0)(63,0){2}{\vector(1,0){12}}
             \put(54,0){\line(1,0){8}}
             \multiput(66,0)(5,0){3}{\line(1,0){1}}
             \put(81,0){\vector(1,0){8}}
           }
           \put(0,60){
                \multiput(30,0)(20,0){2}{\circle*{3}}
                \multiput(93,0)(20,0){2}{\circle*{3}}
             \multiput(34,0)(63,0){2}{\vector(1,0){12}}
             \put(54,0){\line(1,0){8}}
             \multiput(66,0)(5,0){3}{\line(1,0){1}}
             \put(81,0){\vector(1,0){8}}
           }
           \multiput(0,0)(143,0){2}{\circle*{3}}
           \put(3,1.5){\vector(2,1){24}}
           \put(1.8,2.4){\vector(3,4){26.4}}
           \put(1.5,3){\vector(1,2){27}}
           \put(113,0){
             \put(3,13.5){\vector(2,-1){24}}
             \put(1.8,37.6){\vector(3,-4){26.4}}
             \put(1.5,57){\vector(1,-2){27}}
           }
           \put(139,0){\vector(-1,0){135}}
           \put(71.5,63){\HBCenter{\small $\delta_1$}}
           \put(71.5,43){\HBCenter{\small $\delta_2$}}
           \put(71.5,18){\HBCenter{\small $\delta_t$}}
           \put(71.5,-3){\HTCenter{\small $\eta$}}
           \multiput(41.5,30)(60,0){2}{
              \multiput(0,0)(0,-3){3}{\circle*{1}}
            }
            \multiput(71.5,35)(0,-3){3}{\circle*{1}}
            \put(-3,0){\RVCenter{\small $x$}}
            \put(146,0){\LVCenter{\small $y$}}
         }
       \end{picture}
     \end{center}
Then there exists a unique zero relation $\rho$ antiparallel to $\eta$ and $\rho=\sum_{i=1}^{t}\lambda_i\delta_i$ with $\lambda_i\neq 0$ for all $1\leq i\leq t$.
\end{proposition}

\begin{proof}
  We proceed in steps.

  (i)\ \ \emph{Assume that $x_0=x\xrightarrow{\gamma_1} x_1
    \xrightarrow{\gamma_2} x_2 \rightarrow \cdots \rightarrow
    x_{L-1}\xrightarrow{\gamma_L} x_L=y$ is a shortest path in $Q_C$
    and that there exists a minimal relation, $\rho=\sum_{k=1}^{t}\lambda_k\delta_k$, with $\delta_1 = \gamma_{j}\cdots\gamma_{i+1}$}  for some $0\leq i<j\leq L$ then $i=0$ and $j=L$.
  Let $C'$ be the quotient obtained from $C$ by killing all vertices
  except those along the given path.  A minimal relation in $C$ which
  has as non-zero summand the path $\gamma_{j}\cdots\gamma_{i+1}$
  implies that $\gamma_{j}\cdots\gamma_{i+1}=0$ in $C'$.  Observe that
  the quiver $Q'$ of $C'$ is an oriented cycle. By \cite{BMR} the
  algebra $C'$ is again cluster-tilted and hence by Remark~\ref{rem:C_n} there is a minimal zero relation
  $\gamma_L\cdots\gamma_1$ in $C'$. Hence $i=0$ and $j=L$.

  (ii)\ \
  \emph{The shortest paths which are antiparallel to $\eta$ form one homotopy class
  in $C$.}

  Assume otherwise and choose two non-homotopic paths
  $\delta=\delta_m\cdots\delta_1$ and
  $\varepsilon=\varepsilon_n\cdots\varepsilon_1$. Let $C'$ be the
  quotient obtained from $C$ by killing all vertices which are not
  contained in these two paths.  The quiver of $C'$ looks then as
  follows.

   \begin{center}
       \begin{picture}(143,50)
         \put(0,5){
            \multiput(0,0)(0,40){2}{
                \multiput(30,0)(20,0){2}{\circle*{3}}
                \multiput(93,0)(20,0){2}{\circle*{3}}
             \multiput(34,0)(63,0){2}{\vector(1,0){12}}
             \put(54,0){\line(1,0){8}}
             \multiput(66,0)(5,0){3}{\line(1,0){1}}
             \put(81,0){\vector(1,0){8}}
           }
          \multiput(0,20)(143,0){2}{\circle*{3}}
          \put(-6,18){\HBCenter{$x$}}
          \put(149,18){\HBCenter{$y$}}
           \put(3,22){\vector(3,2){24}}
           \put(3,18){\vector(3,-2){24}}
          \put(113,20){
             \put(3,18){\vector(3,-2){24}}
             \put(3,-18){\vector(3,2){24}}
           }
           \put(139,20){\vector(-1,0){135}}
           \put(0,20){
              \put(12,13){\HBCenter{\small $\delta_1$}}
              \put(12,-13){\HTCenter{\small $\varepsilon_1$}}
            }
            \put(143,20){
              \put(-12,13){\HBCenter{\small $\delta_m$}}
              \put(-12,-13){\HTCenter{\small $\varepsilon_n$}}
            }
           \put(71.5,17){\HTCenter{\small $\eta$}}
         }
       \end{picture}
     \end{center}

  The two arms are of length $m$ and $n$ respectively and we shall call
  this quiver $G(m,n)$.
  Since these two paths are non-homotopic,
  we must have two minimal zero relations $\delta=0$ and $\varepsilon=0$.
  The mutation of the quiver $Q_{C'}$ in the vertex $t(\delta_1)$ and then
  killing this vertex gives the quiver $G(m-1,n)$. Observe that we still must
  have that both paths of length $n-1$ and $m$ respectively, are minimal
  zero relations since the algebra is obtained as quotient of $C´$.
  Proceeding this way we get $G(2,2)$ which is occurs as quiver of a
  cluster-tilted algebra of type $D_4$, where the two paths of length $2$
  are non-zero but their sum forms a minimal zero relation. Hence we got
  a contradiction and all paths from $x$ to $y$ must be homotopic.

  (iii)
  \emph{There exists precisely one minimal relation antiparallel to $\eta$.}

  Otherwise choose some minimal relation $\rho_1$ involving the paths
  $\delta_1,\ldots,\delta_b$, that is
  \begin{equation}
  \label{eq:rel1}
  \rho_1=\sum_{i=1}^b\lambda_i\delta_i
  \quad\text{with $\lambda_i\neq 0$ for $1\leq i\leq b$}.
  \end{equation}
  Since all paths are homotopic there exists a
  second relation $\rho_2$  involving some of these paths and possibly more.
  We can assume that
    \begin{equation}
  \label{eq:rel2}
  \rho_2=\sum_{i=a}^c \mu_i \delta_i
  \quad\text{with $\mu_i\neq0$ for $a\leq i\leq c$}
  \end{equation}
  for some
  $1\leq a\leq b\leq c$. If $a=1$ then we can replace $\rho_2$ by
  $\rho_2-\frac{\mu_1}{\lambda_1}\rho_1$. Similarly if $b=c$,
  we replace $\rho_1$ by $\rho_1-\frac{\lambda_b}{\mu_b}\rho_2$. After these
  replacements we get two relations \eqref{eq:rel1} and \eqref{eq:rel2} with
  $1<a\leq b<c$. Now kill all the idempotents involved in the paths
  except those in $\delta_1,\delta_c$. As a quotient we get an algebra $C´$,
  which is cluster-tilted by \cite{BMR} and whose quiver is $G(m,n)$ for some
  $m$ and $n$ with two zero relations, which is impossible by (ii).
\end{proof}

\begin{proposition}
\label{prop:rel-correspond-arrow}
  Let $C$ be a cluster-tilted algebra whose quiver is cyclically oriented. Then the following holds.
  \begin{itemize}
  \item[{\rm(R1)}]
     The arrows of $Q_C$, which occur in some oriented chordless cycle correspond
     bijectively to the minimal relations in any presentation of $C$.
  \item[{\rm (R2)}]
     Let $\eta$ be some arrow of $Q_C$ which occurs in some
     oriented  chordless cycle and let $\delta_1,\ldots,\delta_t$ be the shortest paths
     which are antiparallel to $\eta$. Then the minimal relation corresponding
     to $\eta$ is of the form $\sum_{i=1}^t \lambda_i\delta_i$ with $\lambda_i\neq 0$
     for all $i$. Moreover the quiver restricted to the vertices involved in all the paths
     $\delta_1,\ldots,\delta_t$ looks as shown in Proposition~\ref{prop:clu-tilt-canonical},
     in particular, the paths $\delta_1,\ldots,\delta_t$ share only the endpoints.
  \end{itemize}
\end{proposition}

\begin{proof}
  It follows from Propositions~\ref{prop:cyc-or-quiv} and
\ref{prop:clu-tilt-canonical} that each arrow $\eta\colon
y\rightarrow x$ corresponds to a unique minimal relation $\rho_\eta$
antiparallel to $\eta$.
  Conversely assume now that $\rho$ is a minimal relation. By \cite[Lemma 7.2]{BRS}
  there exists at least one arrow $\eta$ which is antiparallel to $\rho$ and since $C$ has no double arrows there is no other. This shows (R1).

 (R2) follows from Proposition~\ref{prop:cyc-or-quiv} and Proposition~\ref{prop:clu-tilt-canonical}.
\end{proof}

\begin{example}
The following example shows that the hypothesis that $Q_C$ is cyclically oriented is important.
Let $C=\mu_2(k Q)$, the mutation in the vertex $2$ of the path-algebra $k Q$, where the quiver $Q$ is as shown in the following picture.
\begin{center}
  \begin{picture}(240,70)
    \put(0,20){
        \put(-8,40){\RBCenter{$Q:$}}
    \multiput(0,0)(80,0){2}{\circle*{3}}
    \put(40,40){\circle*{3}}
    \qbezier(36.4,38.8)(10,30)(1.2,3.6)
    \put(36.4,38.8){\vector(3,1){0.001}}
    \put(2,2){\vector(1,1){36}}
    \put(42,38){\vector(1,-1){36}}
    \qbezier(78.8,3.6)(70,30)(43.6,38.8)
    \put(78.8,3.6){\vector(1,-3){0.001}}
    \put(4,-1){\vector(1,0){72}}
    \put(-3,0){\RVCenter{\small $1$}}
    \put(40,43){\HBCenter{\small $2$}}
    \put(83,0){\LVCenter{\small $3$}}
    \put(9,31){\HVCenter{\small $\alpha_1$}}
    \put(24,16){\HVCenter{\small $\alpha_2$}}
    \put(71,31){\HVCenter{\small $\beta_1$}}
    \put(56,16){\HVCenter{\small $\beta_2$}}
    \put(40,-6){\HVCenter{\small $\gamma_1$}}
    }
    \put(180,20){
         \put(0,40){\RBCenter{$Q_C:$}}
         \multiput(0,0)(80,0){2}{\circle*{3}}
         \put(40,40){\circle*{3}}
    \put(38,38){\vector(-1,-1){36}}
    \qbezier(36.4,38.8)(10,30)(1.2,3.6)
    \put(1.2,3.6){\vector(-1,-3){0.001}}
    \put(78,2){\vector(-1,1){36}}
    \qbezier(78.8,3.6)(70,30)(43.6,38.8)
    \put(43.6,38.8){\vector(-3,1){0.001}}
    \put(6,0){\vector(1,0){68}}
    \qbezier(6,2)(40,10.5)(74,2)
    \put(74,2){\vector(4,-1){0.1}}
    \qbezier(6,-2)(40,-10.5)(74,-2)
    \put(74,-2){\vector(4,1){0.1}}
    \qbezier(5.6,-3.8)(40,-21)(74.4,-3.8)
    \put(74.4,-3.8){\vector(2,1){0.1}}
    \qbezier(4.2,-6)(40,-36.3)(75.8,-6)
    \put(75.8,-6){\vector(3,2){0.1}}
    \put(-3,0){\RVCenter{\small $1$}}
    \put(40,43){\HBCenter{\small $2$}}
    \put(83,0){\LVCenter{\small $3$}}
    \put(9,31){\HVCenter{\small $\alpha_1^\ast$}}
    \put(26,19){\HVCenter{\small $\alpha_2^\ast$}}
    \put(71,31){\HVCenter{\small $\beta_1^\ast$}}
    \put(53,19){\HVCenter{\small $\beta_2^\ast$}}
    \put(40,12){\HVCenter{\small $\gamma_1$}}
    \put(40,-25){\HVCenter{\small $\gamma_5$}}
     }
  \end{picture}
\end{center}
All four paths $\alpha_i^\ast\beta_j^\ast$ for $i,j=1,2$ are
zero. Hence there are four relations from $3$ to $1$ but five arrows
antiparallel to them.
\end{example}

\subsection{Algebras satisfying {\rm (R1)} and {\rm (R2)}}

In the following we want to show that the  non-zero coefficients
$\lambda_{\eta,i}^C$
appearing in Proposition~\ref{prop:rel-correspond-arrow} do not
change the isomorphism class of the algebra. Therefore it will be
useful to have some short notation.

\begin{definition}
Let $Q$ be a cyclically oriented quiver. Let $Q_1^\cyc$ be the
set of all arrows of $Q$ which belong to a chordless cycle, and for each arrow $\alpha$ let $\rho_\alpha$ be the sum
of all paths antiparallel to $\alpha$ (with coefficients equal to
$1$).   We then denote by
$\standard{Q}$ the algebra $k Q/\langle R\rangle$ with
$R=\{\rho_\alpha\mid\alpha\in Q_1^\cyc\}$.

\end{definition}

\begin{proposition}
\label{prop:deformed} Let $Q$ be a cyclically oriented quiver. Then
each algebra with quiver $Q$ whose relations satisfy {\rm (R1)} and {\rm (R2)}
is isomorphic to the cyclically normalized algebra $\standard{Q}$.
\end{proposition}

\begin{proof}
Let $B$ be an algebra with quiver $Q$ satisfying (R1) and (R2). We shall denote the coefficients appearing in the relation (R2) as follows:
the relation $\rho_\eta$ antiparallel to an arrow $\eta$ is $\rho_\eta=\sum_{i=1}^{t_\eta}\lambda_{\eta,i}^B\delta_{\eta,i}$, where $\delta_{\eta,1},\ldots,\delta_{\eta,t_\eta}$ are the paths antiparallel to $\eta$ and $\lambda_{\eta,i}^B$ are the non-zero coefficients.
If not all coefficients $\lambda_{\eta,i}^B$ are equal to $1$ then we
construct explicitly an algebra $B'$ with the same quiver and which
also satisfies (R1) and (R2) but which has more coefficients equal
to $1$. By induction we hence get the result.

Let $\xi\in Q_1^\cyc$ be an arrow such that the relation
$\rho_\xi=\sum_\delta \lambda^B_{\xi,\delta}\delta$ of $B$ has some
coefficient $\lambda^B_{\xi,\varphi}\neq 1$. Let $\alpha$ be an
arrow of $\varphi$ and construct an admissible cut
$$
\Sigma=\Sigma_x\cup\Sigma_y\cup\Sigma_1\cup\ldots\cup\Sigma_t
$$
containing $\alpha$ as in the proof of Proposition
\ref{prop:adm-cut-exists}. Recall that by construction $\alpha$ belongs to $\Sigma_1$.

Let $B'=kQ/I'$ be the algebra, which is
defined precisely by the same relations than $B$ with the unique exception that
$\lambda^{B'}_{\xi,\gamma}=1$ in $B'$, that is
$\lambda^{B'}_{\eta,\delta}=\lambda^B_{\eta,\delta}$ whenever
$(\eta,\delta)\neq(\xi,\gamma)$.
Define the isomorphism $f:k Q\rightarrow k Q$ by $f(e_z)=e_z$ for
each vertex $z$ and $f(\beta)=f_\beta \beta$ for each arrow $\beta$
where the coefficients $f_\beta$ are defined as follows:
$f_\beta=\lambda^B_{\xi,\varphi}$ for each $\beta\in\Sigma_1$ (that
is, $\beta$ lies in the same component as $\alpha$) and $f_\beta=1$
otherwise.

Then $f(\sum_\xi \lambda^{B'}_{\xi,\delta} \delta) =\sum_\xi
\lambda^{B}_{\xi,\delta}\delta$
and  $f(\rho^{B'}_\eta)=\rho_\eta^B$ for
each $\eta\neq \xi$ since parallel relations in $Q$ which are not
antiparallel to $\xi$ lie in the same component of
$\overline{\Gamma}_x,\overline{\Gamma}_y,\overline{\Gamma}_1,\ldots,\overline{\Gamma}_t$
by Proposition~\ref{prop:cyc-or-quiv3}.
\end{proof}

\begin{corollary}
\label{cor:standard} If $C$ is a cluster-tilted algebra whose quiver
$Q_C$ is cyclically oriented then $C$ is isomorphic to the algebra
$\standard{Q_C}$.
\end{corollary}

\begin{proof}
This is an immediate consequence of Proposition~\ref{prop:rel-correspond-arrow} and Proposition~\ref{prop:deformed}.
\end{proof}


\subsection{Admissible cuts}

An algebra $A$ with connected
quiver $Q$ with no oriented cycles is called \emph{simply connected} if for each presentation
$(Q, I)$ of $A$ the fundamental group $\pi(Q, I)$ is trivial, for precise definitions we
refer to \cite{BG} and \cite{Sko}.

\begin{proposition}
\label{prop:cut-is-ssc}
Let $C$ be a cluster-tilted algebra such that $Q_C$ is cyclically
oriented. Then each quotient of $C$ by an admissible cut is strongly simply
connected.
\end{proposition}

\begin{proof}
We know from Proposition~\ref{prop:cut-directed} that the quiver
$Q_A$ of the cut $A$ of $C$ is directed. Now, by Proposition
\ref{prop:rel-correspond-arrow}, we know that for each full and
convex subalgebra $B'$ of $B$ and any two vertices $x$ and $y$ of
$Q_{B'}$  the paths from $x$ to $y$ in $Q_{B'}$ form a unique
homotopy class. Therefore $B'$ is simply connected. This shows that
$B$ is strongly simply connected.
\end{proof}

\begin{theorem}
\label{thm:implicaciones}
  Let $C$ be a cluster-tilted algebra whose quiver is cyclically
  oriented and let $A$ be a finite-dimensional algebra with $\gldim
  A\leq 2$, such that $\calg{A}$ is a finite dimensional algebra.
  Then $\calg{A}\simeq C$ if and only if $A$ is the quotient of
  $C$ by an admissible cut.

\end{theorem}

\begin{proof}
Assume that $\calg{A}\simeq C$.
We know from \cite{Amiot} that $Q_{\calg{A}}=Q_{\ealg{A}}$. Let $\pi$
be the canonical projection $\pi\colon\calg{A}\rightarrow\ealg{A}$.
Since $\pi$ is an epimorphism
of algebras we infer that if $\rho$ is a minimal relation for $\calg{A}$ then
$\pi(\rho)$ is a minimal relation for $\ealg{A}$.

  By Lemma~\ref{lem:cut-iff}, it suffices to show that $A$ the quotient of
   $\ealg{A}$ by an admissible cut. Take any presentation of $A$ and extend
  it to a presentation of $\ealg{A}$. Since $Q_{\calg{A}}=Q_{\ealg{A}}$ and there are no multiple
  arrows in $\ealg{A}$, no arrow can be
  parallel to a minimal relation in $A$.

  Now, let
  $\delta=\delta_t\delta_{t-1}\cdots\delta_1$ be a chordless
   oriented circle in
  $Q_C$ with $\delta_1\not\in\Phi$, where $\Phi$ is the set of
   arrows of $Q_C$ that are not arrows of $A$.

    By Proposition~\ref{prop:rel-correspond-arrow},
  there exists a minimal relation $\rho$  antiparallel to $\delta_1$,
  which has the path $\gamma=\delta_t\cdots\delta_{2}$ as
  summand. By the above $\mu(\rho)$ cannot be a minimal relation
  for $A$, and therefore at least one of the
  arrows $\delta_2,\ldots,\delta_t$ belongs to $\Phi$.

  Reorder the indices of the cycle such that $\delta_t\in\Phi$.
  Then there exists a minimal relation $\rho$ for $A$ antiparallel
  to $\delta_t$. Since $\ealg{A}$ is a relation-extension of $A$, the
  relation $\rho$ is also a minimal relation for $\ealg{A}$. But
  for $\ealg{A}$ there exists a unique relation antiparallel to
  $\delta_t$, and this unique relation contains
  $\gamma=\delta_{t-1}\cdots\delta_1$ as summand. This shows that none
  of the arrows $\delta_1,\ldots,\delta_{t-1}$ belongs to $\Phi$.
  Altogether, we have proved that of each oriented chordless cycle precisely
  one arrow belongs to $\Phi$.

  For the reverse implication suppose that $A$ is the quotient of  $C$ by an admissible cut,
   and let
$\Phi$ be the set of arrows of $Q_C$ which do not belong to $Q_A$.
By
  definition of admissible cut, each arrow $\gamma\in\Phi$ belongs to
  an oriented cycle and therefore corresponds to a relation
  $\rho_\gamma$ which is antiparallel to $\gamma$ by Proposition~\ref{prop:rel-correspond-arrow}. Therefore, the quiver
  $Q_{\ealg{A}}$ of the relation extension of $A$ is isomorphic to the
  quiver $Q_C$ and also to the quiver $Q_{\calg{A}}$ of $\calg{A}$.
 Then by \cite{BIRS}, Corollary 2.4, we have that $\calg{A}\simeq
 C$.

\end{proof}

\subsection{Dynkin and extended Dynkin case}

We now focus on two cases which are of particular interest, namely
when $C$ is a cluster-tilted algebra of Dynkin type $\AA$, $\DD$ or
$\EE$ or when extended Dynkin type $\tilde{\DD}$ or
$\tilde{\EE}$. The example in the introduction shows
that the following theorem can not be extended to type $\tilde{\AA}$.
Note that if $C$ is a cluster tilted algebra
of type $\tilde{\AA}$ its ordinary quiver is not cyclically oriented,
since $C$ by \cite{ABS}, is the relation extension of representation infinite
tilted algebra of type $\tilde{\AA}$, which is a branch enlargement of a tame
concealed algebra of type $\tilde{\AA}$.

\begin{proposition}
\label{thm:iteratedtilted}
  Let $C$ be a cluster-tilted algebra of Dynkin or extended Dynkin type
  $\Delta$ such that $Q_C$ is cyclically oriented.  Let
  $A$ be a quotient by an admissible cut of $C$ of $\gldim A\leq 2$. Then
  $A$ is derived equivalent to $k \Delta$.
\end{proposition}

\begin{proof}
  If $\calg{A}\simeq C$ then $A$ is the quotient of $C$ by an admissible cut by Proposition~\ref{thm:implicaciones}(a).

  So suppose now that $A$ is a quotient by an admissible cut of $C$.

  We know from Proposition~\ref{prop:cut-is-ssc} that $A$ is strongly
  simply connected. If $\Delta$ is a Dynkin diagram then by \cite{BGZ}
  there exists a quasi-Cartan companion $M$ which is positive
  definite. If $\Delta$ is an extended Dynkin diagram then by
  \cite{Seven}, there exists a quasi-Cartan companion $M$ which is
  positive semi-definite of corank one.

  The proof of \cite[Prop.~4.19]{BFPPT} can be repeated literally to
  show that the Euler form of $A$ is equivalent to $M$.  In the case that $\Delta$ is of Dynkin type, we get that the quadratic form of $A$ is positive definite and $A$ is strongly simply connected by Propostition~\ref{prop:cut-is-ssc}. Hence it follows by \cite{AS} that $A$ is derived equivalent to $k\Delta$. In the case where $\Delta$ is of
  extended Dynkin type, we get that the quadratic form of $A$ is positive semi-definite and $A$ is strongly simply connected by Proposition~\ref{prop:cut-is-ssc}. Hence it follows from \cite{BdP} that
  $A$ is derived equivalent to $k\Delta$.

\end{proof}

\begin{theoremunotres}
Let $A$ be a finite-dimensional algebra with $\gldim
  A\leq 2$, such that $\calg{A}$ has a cyclically oriented quiver.
  Then $\calg{A}$ is cluster tilted of Dynkin or extended Dynkin
  type $\Delta$ if and only if $A$ is derived equivalent to a hereditary algebra
  $H$ of Dynkin or extended Dynkin type $\Delta$.
\end{theoremunotres}

\begin{proof}
  Assume $A$ is derived equivalent to a hereditary algebra  $H$ of type $k \Delta$. Hence there exists
  a tilting complex $T$ such that $A$ is isomorphic to
  $\End_{\Der(H)}(T)$. By \cite[Thm. 1.1]{BFPPT}, the
  algebra $\calg{A}=\End_{\Clu(H)}(T)$ is a cluster-tilted algebra of type $k \Delta$.

  Conversely, consider $\calg{A}$ a cluster tilted algebra of Dynkin or extended Dynkin
  type $k \Delta$, whose quiver is cyclically oriented. It follows by Theorem~\ref{thm:implicaciones} that $A$ is a
 quotient by an admissible cut of $\calg{A}$.
 Applying Proposition~\ref{thm:iteratedtilted}, it follows that
   $A$ is derived equivalent to a hereditary algebra $H$ of type $\Delta$.
\end{proof}

\begin{remark}
  Let $C$ be of minimal infinite type, that is, each quotient by a
  non-zero idempotent is of finite type. Then each {quotient  of
  $C$ by an admissible cut} is either of finite type or tame concealed. Indeed, if $A$ is the quotient of $C$ by an
  admissible cut and $A$ is not finite type then each quotient $A'=A/AeA$  by a non-zero idempotent then $A'$ is
  an admissible cut of
  $C'=C/CeC$. Since $C'$ is of finite type, also $A'$ is of finite type. Now, since $A$ is strongly simply connected, it
   admits a preprojective component, see \cite{D-JAP} and therefore by  \cite{Bo, HaVo} $A$ is tame concealed.
\end{remark}



\begin{thebibliography}{99}
\bibitem{Amiot}{Claire Amiot: {\it Cluster categories for algebras of global dimension 2 and quivers with potential.}  {\tt arXiv:0805.1035}}

\bibitem{ABS}{Ibrahim Assem, Thomas Br\"ustle, Ralph Schiffler: {\it Cluster-tilted algebras as trivial extensions.} Bull. Lond. Math. Soc. 40 (2008), no. 1, 151--162.}

\bibitem{ACT}{Ibrahim Assem, Fl\'avio U. Coelho, Sonia Trepode: {\it Simply connected tame quasi-tilted algebras.} J. Pure Appl. Algebra 172 (2002), no. 2-3, 139--160.}

\bibitem{AO} {Claire Amiot, Steffen Oppermann: {\it Cluster equivalence and graded derived
equivalence.} arXiv:1003.4916.}

\bibitem{AO2}{Claire Amiot, Steffen Oppermann: {\it Algebras of tame acyclic cluster type.}
arXiv:1009.4065v1}

\bibitem{AS}{Ibrahim Assem, Andrzej Skowro\'nski: {\it  Quadratic forms and iterated tilted algebras.} J. Algebra 128 (1990), no. 1, 55--85.}

\bibitem{AS2}{Ibrahim Assem, Andrzej Skowro\'nski: {\it On some classes of simply connected algebras.}
Proc. London Math. Soc. (3) 56 (1988), no. 3, 417--450. }

\bibitem{BGZ}{Michael Barot, Christof Geiss, Andrei Zelevinsky: {\it Cluster algebras of finite type and positive symmetrizable matrices.} J. London Math. Soc. (2) 73 (2006), no. 3, 545--564.}

\bibitem{BFPPT}{Michael Barot, Elsa Fernandez, María In\'es Platzeck, Nilda Isabel Pratti, Sonia Trepode:
{\it From iterated tilted algebras to cluster-tilted algebras.}
Advances in Mathematics
Volume 223, (2010) no. 4, 1468--1494.}

\bibitem{BdP}{Michael Barot, Jos\'e Antonio de la Pe\~na: {\it Derived tubular strongly simply connected algebras.} Proc. Amer. Math. Soc. 127 (1999), no. 3, 647--655.}

\bibitem{Bo}{Klaus Bongartz: {\it Critical simply connected algebras.} Manuscripta Math. 46 (1984), no. 1-3, 117--136.}

\bibitem{BG}{O. Bretscher, P. Gabriel: \sl{The standard from of a
    representation-finite algebra}. Bull. Soc. Math. France 111 (1983),
  no. 1, 21--40.}

\bibitem{BIRS}{Aslak Bakke Buan, Osuma Iyama,
Idun Reiten, David Smith: {\it Mutation of cluster-tilting objects and potentials.}
{\tt ArXiv:0804.3813arXiv:0804.3813}
}

\bibitem{BMR}{Aslak Bakke Buan, Robert J. Marsh, Idun Reiten: {\it Cluster mutation via quiver representations.} Comment. Math. Helv. 83 (2008), no. 1, 143--177.}

\bibitem{BMR2}{Aslak Bakke Buan, Robert J. Marsh, Idun Reiten: {\it  Cluster-tilted algebras of finite representation type.} J. Algebra 306 (2006), no. 2, 412--431.}

\bibitem{BRS}{Aslak Bakke Buan, Idun Reiten, Idun, Ahmet Seven:
{\it Tame concealed algebras and cluster quivers of minimal infinite type.}
J. Pure Appl. Algebra 211 (2007), no. 1, 71--82. }

\bibitem{D-JAP}{Peter Dr\"axler, Jos\'e Antonio de la Pe\~na: {\it On the existence of postprojective components in the Auslander-Reiten quiver of an algebra.}
Tsukuba J. Math. 20 (1996), no. 2, 457--469.}

\bibitem{F}{
Elsa Fern\'andez : {\it Extensiones triviales y \'algebras inclinadas iteradas.} Ph.D. Thesis, Universidad
Nacional del Sur, Argentina (1999).}

\bibitem{FP}{Elsa Fern\'andez, María In\'es Platzeck: {\it Isomorphic trivial extensions of finite dimensional algebras. J. Pure Appl. Algebra 204} (2006), no. 1, 9--20.}

\bibitem{FZ}{S.~Fomin, A.~Zelevinsky: {\sl Cluster algebras I:
      Foundations}. J. Amer. Math. Soc. 15 (2002), no. 2, 497--529.}

\bibitem{HaVo}{Dieter Happel, Dieter Vossieck:
{\it Minimal algebras of infinite representation type with preprojective component.}
Manuscripta Math. 42 (1983), no. 2-3, 221--243. }

\bibitem{Seven}{Ahmet I. Seven: {\it Cluster Algebras and Semipositive Symmetrizable Matrices.} {\tt arXiv:0804.1456}}

\bibitem{Sko}{
A.~Skowro\'nski: {\it Simply connected algebras and Hochschild
  cohomologies.\/}  Representations of algebras (Ottawa, ON, 1992),
431--447, CMS Conf. Proc., 14, Amer. Math. Soc., Providence, RI,
1993.}


\end{thebibliography}
\end{document}